\documentclass[a4paper,12pt, leqno]{amsart}
\usepackage{amsmath,amsthm,amsfonts,amssymb,graphicx,mathrsfs,esint,longtable,mathtools, enumerate}
\usepackage[colorlinks=true, linkcolor=red, urlcolor=red, citecolor=green]{hyperref}
\usepackage{cleveref}
\usepackage{color}

\newcommand{\R}{{\mathbb R}}

\def\XXint#1#2#3{{\setbox0=\hbox{$#1{#2#3}{\int}$ }
\vcenter{\hbox{$#2#3$ }}\kern-.6\wd0}}

\newcommand{\mres}{\mathbin{\vrule height 1.6ex depth 0pt width
0.13ex\vrule height 0.13ex depth 0pt width 1.3ex}}

\newcommand{\B}{\mathbb{B}}

\newcommand{\Z}{\mathbb{Z}}

\newcommand{\Ha}{\mathcal{H}}
\newcommand{\modd}{\textnormal{mod}}
\newcommand{\capp}{\textnormal{cap}}

\newcommand{\id}{\text{id}}

\newtheorem{theorem}{\textbf{THEOREM}}[section]
\newtheorem{lemma}[theorem]{\textsc{Lemma}}
\newtheorem{proposition}[theorem]{\textsc{Proposition}}
\theoremstyle{definition}
\newtheorem{question}[theorem]{\textsc{Question}}
{\theoremstyle{remark} }
\def\charfn_#1{{\raise1.2pt\hbox{$\chi_{\kern-1pt\lower3pt\hbox{{$\scriptstyle#1$}}}$}}}
\def\leq{\leqslant }
\def\geq{\geqslant }

\def\XXint#1#2#3{{\setbox0=\hbox{$#1{#2#3}{\int}$}
\vcenter{\hbox{$#2#3$}}\kern-.5\wd0}}

\begin{document}

\title{On the duality of moduli in arbitrary codimension}
\author{Atte Lohvansuu} 
\let\thefootnote\relax
\thanks{The author was supported by the Academy of Finland, grant no. 308659, and also by the Vilho, Yrj\"o and Kalle V\"ais\"al\"a foundation.}
\begin{abstract} We study the duality of moduli of $k$- and $(n-k)$-dimensional slices of euclidean $n$-cubes, and establish the optimal upper bound 1. 
\end{abstract}
\maketitle
\section{Introduction and the main result}
Suppose $D\subset\R^2$ is a Jordan domain, whose boundary is divided into four segments $\zeta_1, \ldots, \zeta_4$, in cyclic order. Let $\Gamma(\zeta_1, \zeta_3; D)$ be the family of all paths of $D$ that connect $\zeta_1$ and $\zeta_3$. Then for every $1<p<\infty$
\begin{equation}\label{eq:polku1}
(\modd_p\Gamma(\zeta_1, \zeta_3; D))^{1/p}(\modd_{q}\Gamma(\zeta_2, \zeta_4; D))^{1/q}=1.
\end{equation}
Here $q=\frac{p}{p-1}$ and the $p$-modulus of a path family $\Gamma$ is defined by
\[
\modd_p\Gamma=\inf_\rho\int_D\rho^p\, d\Ha^2,
\]
where the infimum is taken over all positive Borel-functions $\rho$ with
\[
\int_\gamma\rho\, ds\geq 1
\]
for every locally rectifiable path $\gamma\in\Gamma$. 

For conformal moduli, that is $p=2=q$, the duality \eqref{eq:polku1} was already known to Beurling and Ahlfors, see e.g. \cite[Lemma 4]{AhlforsBeurling1950} and \cite[Ch. 14]{AhlforsSarioRS}, although instead of moduli they considered their reciprocals, called  \emph{extremal lengths}. For general $p$ the identity \eqref{eq:polku1} follows from the results of \cite{Ziemer1967}. It has found applications in connection with uniformization theorems \cite{Rajala2017, Ikonen2019} and Sobolev extension domains \cite{Zhang2020}.  

The \emph{duality of moduli} phenomenon \eqref{eq:polku1} is also present in euclidean spaces \cite{FreedmanHe1991, Gehring1962, Ziemer1967} of higher dimension and sufficiently regular metric spaces \cite{JonesLahti2019, Lohvansuu2020, LohvansuuRajala2018}. For example, in \cite{Ziemer1967} it is shown that
\begin{equation}\label{eq:polku2}
(\modd_p\Gamma(E, F; G))^{1/p}(\modd_{q}\Gamma^*(E, F; G))^{1/q}=1,
\end{equation}
where $G\subset\R^n$ is open and connected, $E$ and $F$ are disjoint, compact and connected subsets of $G$ and $\Gamma^*(E, F; G)$ is the set of all compact sets of $G$ that separate $E$ from $F$. The modulus of separating sets is a natural generalization of the definition of the path modulus. See Section \ref{section:preli} for definitions of moduli and other concepts appearing in the introduction.

Separating sets are generally of codimension $1$, so \eqref{eq:polku1} and \eqref{eq:polku2} deal with objects of either dimension or codimension 1. In fact, this is a common theme in all of the results cited above. However, an observation by Freedman and He (see the discussion after Theorem 2.5 in \cite{FreedmanHe1991}) hints that a similar duality result could be true for objects of higher (co)dimension as well. In this paper we explore this question in the setting of cubes of $\R^n$. 

Moduli of higher (co)dimensional objects have appeared in \cite{HeinonenWu2010, PankkaWu2014}, where the nonexistence of quasisymmetric parametrizations of certain spaces was established. Indeed, one of the main motivations for studying more general moduli is finding tools to approach parametrization problems in higher dimensions.

Our first problem is defining suitable classes of $k$- and $(n-k)$-dimensional objects, since simple descriptions such as ``connecting paths" or ``separating surfaces" do not seem to exist. We follow \cite{FreedmanHe1991} and define the objects as representatives of certain relative homology classes. For example, in the context of \eqref{eq:polku1} we can think of the paths of $\Gamma(\zeta_1, \zeta_3; D)$ as singular relative cycles, that are representatives of either generator of $H_1(D, \zeta_1\cup\zeta_3)\simeq\Z$. Since we also want to integrate over the chains, we need to assume some regularity. For this reason we will consider Lipschitz chains instead of singular chains.

Let $Q\subset \R^n$ be a compact set homeomorphic to the closed unit $n$-cube $I^n$. Fix a homeomorphism $h: Q\rightarrow I^n$ and an integer $0<k<n$ and let 
\[
A=h^{-1}(\partial I^k\times I^{n-k})\text{ and } B=h^{-1}(I^k\times\partial I^{n-k}).
\] 
Then $A$ and $B$ are $(n-1)$-dimensional submanifolds of $\partial Q$ with $\partial Q=A\cup B$ and $\partial A=A\cap B=\partial B$. We assume that $A, B$ and $Q$ are locally Lipschitz neighborhood retracts. This includes triples $(Q, A, B)$ that are smooth or polygonal, and cubes that are images of the standard cube under biLipschitz automorphisms of $\R^n$.

We denote the Lipschitz homology groups by $H^L_*$. We consider only groups with integer coefficients. This notation should not be confused with the Hausdorff measures, which are denoted by $\Ha^*$. Note that
\[
H_k^L(Q, A)\simeq\Z\simeq H^L_{n-k}(Q, B),
\]
since the same is true for singular homology, and the two homology theories are equivalent for pairs of locally Lipschitz retracts (see Lemma \ref{lemma:liphomo}). 

Let $\Gamma_A$ (resp. $\Gamma_B$) be the collection of the images of relative Lipschitz $k$-cycles of $Q-B$ that generate $H_k^L(Q, A)$ (($n-k$)-cycles of $Q-A$ that generate $H_{n-k}^L(Q, B)$). Define
\[
\modd_p\Gamma_A:=\inf_\rho\int_Q\rho^p\, d\Ha^n,
\]
where the infimum is taken over positive Borel-functions $\rho$, for which
\[
\int_S\rho\, d\Ha^k\geq 1
\]
for every $S\in\Gamma_A$. The moduli $\modd_p\Gamma_B$ are defined analogously. In this paper we will prove the following upper bound.
\begin{theorem}\label{thm:main} 
For every $1<p<\infty$
\[
(\modd_p\Gamma_A)^{1/p}(\modd_{q}\Gamma_B)^{1/q}\leq 1,
\]
where $q=\frac{p}{p-1}$.
\end{theorem}
It is unknown, whether Theorem \ref{thm:main} holds with an equality. We will prove Theorem \ref{thm:main} in Section \ref{sec:proof}. A similar result for de Rham cohomology classes, with an equality, is proved in the setting of Riemannian manifolds in pages 212-213 of \cite{FreedmanHe1991}.

The assumption on $Q$, $A$ and $B$ being locally Lipschitz neighborhood retracts can be relaxed. The proof of Theorem \ref{thm:main} only requires that there exists a pair of Lipschitz chains that generate $H_k(Q, A)$ and $H_{n-k}(Q, B)$. The assumption on retracts was chosen for its simplicity and its use in \cite{GMT}. It is also likely that such minimal assumptions on the upper bound of Theorem \ref{thm:main} are not sufficient for the corresponding lower bound. We will discuss the lower bound in Section \ref{section:remarks}.

In light of the results of \cite[Ch. 4]{GMT}, it would be interesting to know whether analogues of Theorem \ref{thm:main} hold for homology classes of integral currents. 
\section{Definitions}\label{section:preli}
\subsection{Lipschitz homology}
Let us recall the definition and basic properties of the integral homology groups. See e.g. \cite{DoldAT, HatcherAT} or other texts on basic algebraic topology for more comprehensive treatment.

 For an integer $k\geq 0$ the \emph{standard $k$-simplex} $\Delta_k$ is the convex hull of the standard unit vectors $e_0, \ldots, e_k$ of $\R^{k+1}$. Given a metric space $(X, d)$, a \emph{singular $k$-simplex} is a continuous map from $\Delta_k$ to $X$. Finite formal linear combinations
\[
\sigma=\sum_ik_i\sigma_i
\]
of singular $k$-simplices $\sigma_i$ with integer coefficients $k_i$ are called \emph{singular $k$-chains}. Singular $k$-chains of $X$ form a free abelian group denoted by $C_k(X)$. The \emph{boundary} $\partial\sigma$ of a singular $k$-simplex $\sigma$ is the singular $(k-1)$-chain
\[
\partial\sigma=\sum_{i=0}^k(-1)^i\sigma\circ F^i_k,
\]
where $F^i_k: \Delta_{k-1}\rightarrow \Delta_k$ is the unique linear map that maps each $e_j$ to $e_j$ for $j<i$ and to $e_{j+1}$ for $j\geq i$. For singular $0$-simplices we set $\partial\sigma=0$. The boundary defines a collection of homomorphisms $\partial: C_k(X)\rightarrow C_{k-1}(X)$, all denoted by the same symbol $\partial$. Then $\partial\partial=0$. 

The \emph{image} of a singular $k$-simplex $\sigma$ is the compact set $|\sigma|=\sigma(\Delta_k)$. The image of a $k$-chain $\sigma=\sum_ik_i\sigma_i$ is the compact set $|\sigma|=\bigcup_i|\sigma_i|$.

Given a subspace $Y\subset X$, we identify each singular simplex $\sigma$ of $Y$ with the singular simplex $i_Y\circ\sigma$ of $X$, where $i_Y: Y\hookrightarrow X$ is the inclusion map. We define the groups of \emph{relative chains} by
\[
C_k(X, Y):=\frac{C_k(X)}{C_k(Y)},
\]
with the convention $C_k(X, \emptyset)=C_k(X)$. The boundary map induces homomorphisms $\partial: C_k(X, Y)\rightarrow C_{k-1}(X, Y)$, which are again denoted by the same symbol. A chain $\sigma\in C_k(X)$ is called a \emph{cycle} relative to $Y$, if $\partial\sigma\in C_{k-1}(Y)$, or simply a relative cycle if the choice of $Y$ is clear from the context. Similarly, $\sigma$ is called a relative \emph{boundary} if $\sigma=\partial\sigma'+\sigma''$, where $\sigma'\in C_{k+1}(X)$ and $\sigma''\in C_k(Y)$. 

The \emph{singular relative homology groups} of the pair $(X, Y)$ are the quotient groups
\[
H_k(X, Y):=\frac{\mathrm{ker}(\partial: C_k(X, Y)\rightarrow C_{k-1}(X, Y))}{\mathrm{im}(\partial: C_{k+1}(X, Y)\rightarrow C_k(X, Y))}.
\]
The homology groups of $X$ are the groups $H_k(X):=H_k(X, \emptyset)$. The homology class of a (relative) chain $\sigma$ is denoted by $[\sigma]$. The homology classes of $H_k(X, Y)$ are represented by relative $k$-cycles, and two relative $k$-cycles define the same class if and only if their difference is a relative boundary.

If $X'$ is another metric space with a subset $Y'$, and $f: X\rightarrow X'$ is a continuous map with $f(Y)\subset Y'$, we denote by $f_*$ the induced homomorphisms $f_*: C_k(X, Y)\rightarrow C_k(X', Y')$, and also the homomorphisms $f_*: H_k(X, Y)\rightarrow H_k(X', Y')$. These are given by $f_*\sigma=f\circ\sigma$ for singular simplices, $f_*\sum_ik_i\sigma_i=\sum_ik_if_*\sigma_i$ for chains and $f_*[\sigma]=[f_*\sigma]$ for homology classes.

Given a continuous homotopy $H: X\times I\rightarrow X'$ with $H(Y\times I)\subset Y'$, there exists a sequence of homomorphisms 
\[
P: C_k(X, Y)\rightarrow C_{k+1}(X', Y'),
\]
such that
\begin{equation}\label{eq:homotopyformula}
H_{1*}-H_{0*}=P\partial+\partial P.
\end{equation}
Here $H_t(x)=H(x, t)$. Formula \eqref{eq:homotopyformula} is called the \emph{homotopy formula}.

A continuous $f: X\rightarrow Y$ is called a \emph{retraction} if $f\circ i_Y=\id_Y$. The set $Y$ is then called a \emph{retract} of $X$. If $Y$ is a retract of one if its neighborhoods in $X$, it is called a \emph{neighborhood retract}.

The corresponding objects in the Lipschitz category are obtained by replacing each occurrence of ``singular" or ``continuous" with ``Lipschitz". The homotopies involved in these definitions are then required to be Lipschitz with respect to the metric $d((x, t), (x', t'))=d(x, x')+|t-t'|$. We denote the groups of Lipschitz chains by $C_*^L(X, Y)$ and the Lipschitz homology groups by $H_*^L(X, Y)$. We define \emph{locally} Lipshitz objects similarly. However, due to compactness there is often no difference between the corresponding objects of Lipschitz and locally Lipschitz categories. 

\begin{lemma}\label{lemma:liphomo}
Let $Y\subset X\subset\R^n$ be locally Lipschitz neighborhood retracts. Then the inclusions
\[
i: C_*^L(X, Y)\hookrightarrow C_*(X, Y)
\]
induce isomorphisms on homology.
\end{lemma}
Lemma \ref{lemma:liphomo} follows from a more general result \cite[Cor. 11.1.2]{Riedweg2011}, which holds for pairs of \emph{locally Lipschitz contractible} metric spaces. It is straightforward to show that the existence of locally Lipschitz neighborhood retractions implies locally Lipschitz contractibility.

\subsection{Modulus} Given a $1<p<\infty$ and a family $\mathcal{M}$ of Borel measures of $\R^n$, the \emph{$p$-modulus} of $\mathcal{M}$ is the number
\begin{equation}\label{eq:modulidefinition}
\modd_p\mathcal{M}:=\inf_{\rho}\int_{\R^n}\rho^p\, d\Ha^n,
\end{equation}
where the infimum is taken over all Borel functions $\rho: \R^n\rightarrow [0, \infty)$ with
\begin{equation}\label{eq:moduliehto}
\int_{\R^n}\rho\, d\nu\geq 1
\end{equation}
for every $\nu\in\mathcal{M}$. Such functions are called \emph{admissible} for $\mathcal{M}$. If there exists a subfamily $\mathcal{N}\subset\mathcal{M}$ such that $\modd_p\mathcal{N}=0$ and \eqref{eq:moduliehto} holds for all $\nu\in\mathcal{M}-\mathcal{N}$, we say that $\rho$ is \emph{$p$-weakly admissible} or simply \emph{weakly admissible} if the choice of $p$ is clear from the context. It follows that the infimum in \eqref{eq:modulidefinition} does not change if we take it over $p$-weakly admissible functions instead. Let us list some useful properties of the modulus. 
\begin{lemma}\label{lemma:modulilemma}
Let $\mathcal{M}$ be a collection of Borel measures of $\R^n$. Let $1<p<\infty$.
\begin{enumerate}[i)]
\item If $\rho_i$ are $p$-integrable Borel functions that converge to a function $\rho$ in $L^p$, there exists a subsequence $(\rho_{i_j})_j$ for which 
\[
\int_{\R^n}\rho_{i_j}\, d\nu\overset{j\rightarrow\infty}{\longrightarrow}\int_{\R^n}\rho\, d\nu
\]
for almost every $\nu\in\mathcal{M}$. In particular, Borel representatives of $L^p$-limits of admissible functions are weakly admissible.
\item If $\modd_p\mathcal{M}<\infty$, then
\[
\modd_p\mathcal{M}=\int_{\R^n}\rho^p\, d\Ha^n
\]
for a weakly admissible minimizer $\rho$, unique up to sets of $\Ha^n$-measure zero. Moreover,
\[
\modd_p\mathcal{M}\leq \int_{\R^n}\phi\rho^{p-1}\, d\Ha^n
\]
for any other $p$-integrable weakly admissible $\phi$.
\item If $\mathcal{M}=\bigcup_{i=1}^\infty\mathcal{M}_i$ with $\mathcal{M}_i\subset\mathcal{M}_{i+1}$ for all $i$, then
\[
\modd_p\mathcal{M}=\lim_{i\rightarrow\infty}\modd_p\mathcal{M}_i.
\]
\end{enumerate}
\end{lemma}
Claim $i)$ is often referred to as \emph{Fuglede's lemma}. Proofs for $i)$ and the first part of $ii)$ can be found in \cite[Thm. 3]{Fuglede1957}. The second part of $ii)$ and $iii)$ are generalizations of \cite[Lemma 5.2]{LohvansuuRajala2018} and \cite[Lemma 2.3]{Ziemer1969}, respectively. The same proofs apply.

In this paper we abbreviate 
\[
\modd_p\Gamma_A=\modd_p\{\Ha^k\mres S\ |\ S\in\Gamma_A\},
\]
and 
\[
\modd_q\Gamma_B=\modd_q\{\Ha^{n-k}\mres S^*\ |\ S^*\in\Gamma_B\}.
\]

\subsection{Rectifiable sets}
A subset of $\R^n$ is \emph{$k$-rectifiable} if it is covered by the image of a subset of $\R^k$ under a Lipschitz map. A subset of $\R^n$ is \emph{countably $k$-rectifiable} if $\Ha^k$-almost all of it is contained in a countable union of $k$-rectifiable sets. 

See e.g. \cite{GMT, SimonGMT} for basic theory on rectifiable sets. Note that the definition of countable rectifiability in \cite[3.2.14]{GMT} is slightly different from ours. 

Let us record some useful facts on rectifiable sets. The following Fubini-type lemma is an application of \cite[3.2.23]{GMT} and \cite[2.6.2]{GMT}. 
\begin{lemma}\label{lemma:fubini}
Suppose $S^*$ is a countably $k$-rectifiable subset of $\R^n$ and $S$ is a countable union of $l$-rectifiable subsets of $\R^m$. Then $S^*\times S$ is a countably $(k+l)$-rectifiable subset of $\R^{n}\times\R^m$, and 
\[
\int_{S^*\times S}g(x, y)\, d\Ha^{k+l}(x, y)=\int_{S^*}\int_{S} g(x, y)\, d\Ha^l(y)\, d\Ha^{k}(x)
\]
for any positive Borel function $g$ on $\R^n\times \R^m$.
\end{lemma}
Lemma \ref{lemma:fubini} is not true for general countably $k$-rectifiable sets $S$, see \cite[3.2.24]{GMT}. The second tool we need is the coarea formula, see e.g. \cite[12.7]{SimonGMT}.
\begin{lemma}\label{lemma:coarea}
Suppose $m\leq k$. Let $S$ be a countably $k$-rectifiable subset of $\R^n$ and let $u: S\rightarrow \R^m$ be locally Lipschitz. Then
\begin{equation}\label{eq:rectcoarea}
\int_{\R^m}\int_{u^{-1}(z)}g\, d\Ha^{k-m}\, d\Ha^m(z)=\int_SgJ^S_u\, d\Ha^k
\end{equation}
for every positive Borel function $g$ on $S$.
\end{lemma}
Let us define the jacobian $J^S_u$ appearing in \eqref{eq:rectcoarea}. Details can be found in \cite[§12]{SimonGMT}. Suppose first, that $S$ is an embedded $C^1$ $k$-submanifold (without boundary) of $\R^n$. Then $u$ is differentiable at $\Ha^k$-almost every $x\in S$. Fix such an $x$, and let $\{E_1, \ldots, E_k\}$ be an orthonormal basis for the tangent space of $S$ at $x$. Let $Du(x)$ be the jacobian matrix of $u$ at $x$ with respect to standard bases of $\R^n$ and $\R^m$. We set
\[
J^S_u(x):=\sqrt{\det(d^Su(x)d^Su(x)^t)},
\]
where $d^Su(x)$ is the matrix with columns $Du(x)E_i$. It can be shown that $J^S_u(x)$ does not depend on the choice of the basis $\{E_i\}$.

More generally, every countably $k$-rectifiable set $S$ can be expressed as a disjoint union $S=\bigcup_{i=0}^\infty M_i$, where $\Ha^k(M_0)=0$ and each $M_i$ for $i\geq 1$ is contained in an embedded $C^1$ $k$-submanifold $N_i$ of $\R^n$. Given an $x\in M_i$ with $i\geq 1$, we set
\[
J^S_u(x):=J^{N_i}_u(x).
\]
Then $J^S_u$ is well defined $\Ha^k$-almost everywhere on $S$. It can be shown that $J^S_u$ does not depend on the decomposition $S=\bigcup_{i=0}^\infty M_i$, up to sets of $\Ha^k$-measure zero.
\section{Proof of Theorem \ref{thm:main}}\label{sec:proof}
Given any set $S\subset\R^n$ and a vector $y\in\R^n$ we denote 
\[
S_y=\{x+y\ |\ x\in S\}
\]
and 
\[
N_\varepsilon(S)=\{x\ |\ d(x, S)<\varepsilon\}.
\]
Denote by $\Gamma_A^*$ the collection of ($n-k$)-rectifiable subsets $S^*$ of $Q-A$, such that the homomorphism
\[
i_*: H_k^L(Q-S^*, A)\rightarrow H_k^L(Q, A)
\]
induced by inclusion is trivial. Lemma \ref{lemma:leikkauslemma} below implies that $\Gamma_B\subset\Gamma_A^*$. Every set $S^*\in\Gamma_A^*$ intersects with every $S\in\Gamma_A$ in a nonempty set. To see this, note that if $|\sigma|\cap S^*$ is empty for some Lipschitz cycle $\sigma\in C_k(Q)$ relative to $A$, then $[\sigma]=i_*[\sigma]=0$ in $H_k^L(Q, A)$ by the definition of $\Gamma_A^*$.

We abbreviate
\[
\modd_q\Gamma_A^*:=\modd_q\{\Ha^{n-k}\mres S^*\ |\ S^*\in\Gamma_A^*\}.
\]
Theorem \ref{thm:main} is then implied by the following more general result.
\begin{theorem}\label{thm:main2}
For every $1<p<\infty$ 
\[
(\modd_p\Gamma_A)^{1/p}(\modd_q\Gamma_A^*)^{1/q}\leq 1,
\]
where $q=\frac{p}{p-1}$.
\end{theorem}
The rest of this section is focused on the proof of Theorem \ref{thm:main2}.

For each $\delta>0$ let $\Gamma_A^\delta$ be the subcollection of $\Gamma_A$ consisting of those sets whose distance to $B$ is at least $100\delta$. The subcollections $\Gamma_A^{*\delta}$ are defined analogously. In light of $iii)$ of Lemma \ref{lemma:modulilemma}, it suffices to show that
\begin{equation}\label{eq:roe1}
(\modd_p\Gamma^\delta_A)^{1/p}(\modd_{q}\Gamma^{*\delta}_A)^{1/q}\leq 1
\end{equation}
for all $\delta$. Fix a $\delta$ for the rest of the proof. We may assume without loss of generality that the moduli in question are nonzero and the collections $\Gamma_A^\delta$ and $\Gamma_A^{*\delta}$ are nonempty.

The following intersection property of the elements of $\Gamma_A$ and $\Gamma_A^*$ forms the topological core of Theorem \ref{thm:main2}.
\begin{proposition}\label{prop:leikkaus}
The intersection $S_z\cap S^*$ is nonempty for every $S\in\Gamma_A^\delta$, $S^*\in\Gamma_A^{*\delta}$ and $|z|< 10\delta$.
\end{proposition}
We postpone the proof to Subsection \ref{subsec:topology}.

Let $S\in \Gamma^{\delta}_A$. Observe that the map
\begin{equation}\label{eq:distribuutio}
g\mapsto\int_{S} g\, d\Ha^{k}
\end{equation}
is a distribution in $\R^n$. Thus we have by \cite[4.1.2]{GMT} that
\begin{equation}\label{eq:konvoluutio}
\int_Q\phi^{S}_\varepsilon g\, d\Ha^n \overset{\varepsilon\rightarrow 0}{\longrightarrow}\int_{S} g\, d\Ha^k
\end{equation}
for every smooth compactly supported function $g$, where
\[
\phi_\varepsilon^{S}(x):=\int_{S}\phi_\varepsilon(x-y)\, d\Ha^k(y)
\]
is the \emph{convolution} of the distribution \eqref{eq:distribuutio} with respect to a smooth kernel $\phi$. That is, $\phi_\varepsilon(x)=\varepsilon^{-n}\phi(\varepsilon^{-1}x)$ and $\phi$ is a positive smooth function on $\R^n$ that vanishes outside the unit ball $\B^n$ and satisfies $\int_{\B^n}\phi\, d\Ha^n=1$.

Smoothness is convenient for avoiding tedious technicalities, but to see the geometry behind the arguments that follow, the reader is encouraged to repeat the proof with the nonsmooth kernel $\phi=|\B^n|^{-1}\chi_{\B^n}$. 

Theorem \ref{thm:main2} follows via \eqref{eq:roe1} from the following proposition.
\begin{proposition}\label{lemma:taikalemma}
The convolution $\phi^{S_z}_\varepsilon$ is admissible for $\Gamma_A^{*\delta}$ for all $\varepsilon<\delta$ and all $|z|<\delta$.
\end{proposition}
\begin{proof}
Fix an $\varepsilon<\delta$ and a set $S^*\in\Gamma_A^{*\delta}$. Let $z=0$ for now. By Lemma \ref{lemma:fubini}
\begin{align*}
\int_{S^*}\phi_\varepsilon^S(x)\, d\Ha^{n-k}(x)&=\int_{S^*}\int_S\phi_\varepsilon(x-y)\, d\Ha^k(y)d\Ha^{n-k}(x)\\
&=\int_{S^*}\int_{S\cap N_\varepsilon(S^*)}\phi_\varepsilon(x-y)\, d\Ha^k(y)d\Ha^{n-k}(x)\\
&=\int_{(S^*\times S)\cap \{|x-y|<\varepsilon\}}\phi_\varepsilon(x-y)\, d\Ha^n(x, y).
\end{align*}
Now we can apply the coarea formula (Lemma \ref{lemma:coarea}) on the map $u(x, y)=x-y$ to obtain
\begin{equation}\label{eq:apu1}
\int_{S^*}\phi_\varepsilon^S(x)\, d\Ha^{n-k}(x)\geq \int_{\varepsilon\B^n}\int_{(S^*\times S)\cap\{x-y=w\}}\phi_\varepsilon(x-y)\, d\Ha^0d\Ha^n(w) 
\end{equation}
since $J_u^{S^*\times S}\leq 1$. To see this, note for any $(n-k)$- and $k$-dimensional embedded $C^1$ submanifolds $N^*$ and $N$ of $\R^n$ the matrix $d^{N^*\times N}u$ consists of unit column vectors. Thus $J^{N^*\times N}_u\leq 1$. It follows that $J_u^{S^*\times S}\leq 1$ as well, since it can be computed via $J_u^{M_i^*\times M_j}$ with $i, j\geq 1$, where $S^*=\bigcup_{i=0}^\infty M_i^*$ and $S=\bigcup_{i=0}^\infty M_j$ are decompositions of $S^*$ and $S$ as in the discussion following Lemma \ref{lemma:coarea}. Note that the sets $M_0^*\times S$ and $S^*\times M_0$ have zero $\Ha^n$-measure by Lemma \ref{lemma:fubini}. 

Finally, we apply Proposition \ref{prop:leikkaus} on \eqref{eq:apu1} and obtain 
\[
\int_{S^*}\phi_\varepsilon^S(x)\, d\Ha^{n-k}(x)\geq \int_{\varepsilon\B^n}\phi_\varepsilon(w)\, d\Ha^n(w)=1.
\]
The proof in the case of general $z$ reduces to the case $z=0$ via
\begin{equation}\label{eq:zsiirto}
\phi^{S_{z}}_\varepsilon(x)=\phi^S_\varepsilon(x-z),
\end{equation} 
since Proposition \ref{prop:leikkaus} can still be applied.
\end{proof}
\begin{proof}[Proof of Theorem \ref{thm:main2}]
The $q$-modulus of $\Gamma_A^{*\delta}$ is finite by Proposition \ref{lemma:taikalemma}. Let $\rho$ be the unique weak minimizer of $\modd_{q}\Gamma_A^{*\delta}$ given by $ii)$ of Lemma \ref{lemma:modulilemma}. We may assume that $\rho$ vanishes in $N_{10\delta}(A)$ and is defined as zero outside $Q$. Let $g_r$ be the smooth convolution
\[
g_r(x):=\int_{r\B^n}\rho^{q-1}(x+y)\phi_r(y)\, d\Ha^n(y).
\]
Let $S\in\Gamma_A^\delta$ and let $\varepsilon<\delta$. Proposition \ref{lemma:taikalemma} and $ii)$ of Lemma \ref{lemma:modulilemma} imply 
\[
\modd_{q}\Gamma_A^{*\delta}\leq\int_Q\phi^{S_z}_\varepsilon\rho^{q-1}\, d\Ha^n
\]
for all $|z|<\delta$ and $S\in\Gamma_A^\delta$. Note that the product $\phi_\varepsilon^{S_z}\rho^{q-1}$ vanishes in $N_{10\delta}(\partial Q)$, so by \eqref{eq:zsiirto} and a change of variables
\[
\modd_q\Gamma_A^{*\delta}\leq \int_Q\phi_\varepsilon^S(x)\rho^{q-1}(x+z)\, d\Ha^n(x)
\]
for all $|z|<\delta$. Multiplying both sides by $\phi_r(z)$ and integrating over $z$ yields 
\[
\modd_q\Gamma_A^{*\delta}\leq\int_Q\phi_\varepsilon^Sg_r\, d\Ha^n
\]
by Fubini's theorem. Letting $\varepsilon\rightarrow 0$ and then $r\rightarrow 0$ yields
\begin{align*}
\modd_{q}\Gamma_A^{*\delta}\leq\int_S\rho^{q-1}\, d\Ha^k
\end{align*} 
for $\modd_p$-almost every $S\in\Gamma_A^\delta$ by \eqref{eq:konvoluutio} and $i)$ of Lemma \ref{lemma:modulilemma}. Thus 
\[
\frac{1}{\modd_{q}\Gamma_A^{*\delta}}\rho^{q-1}
\]
is weakly admissible for $\Gamma_A^\delta$, so 
\[
\modd_p\Gamma_A^\delta\leq (\modd_{q}\Gamma_A^{*\delta})^{1-p},
\]
which is a rearrangement of \eqref{eq:roe1}.
\end{proof}
\subsection{Topological lemmas}\label{subsec:topology}
In this subsection we complete the proof of Theorem \ref{thm:main} by proving Proposition \ref{prop:leikkaus} and showing that $\Gamma_B\subset\Gamma_A^*$. These are implied by the following two lemmas.
\begin{lemma}\label{lemma:siirtolemma}
Suppose $S\in\Gamma_A^\delta$ and $|y|<10\delta$. Then there exists a singular relative cycle $\sigma_y$, such that it generates $H_k(Q, A)$ and its image coincides with $S_y$ outside $N_{100\delta}(A)$.
\end{lemma}
\begin{lemma}\label{lemma:leikkauslemma}
Suppose $\sigma_A$ and $\sigma_B$ are relative singular chains that generate nontrivial elements of $H_k(Q, A)$ and $H_{n-k}(Q, B)$, respectively. Then $|\sigma_A|\cap|\sigma_B|$ is nonempty.
\end{lemma}
\begin{proof}[Proof of Lemma \ref{lemma:siirtolemma}]
The lemma follows from the homotopy formula \eqref{eq:homotopyformula}. By the definition of $\Gamma_A$ there is a relative cycle $\sigma$ that generates $H_k(Q, A)$ and has $S$ as its image. By applying barycentric subdivision multiple times, if necessary, we may assume that $\sigma$ splits into $\sigma=\sigma_1+\sigma_2$, where $|\sigma_1|\subset N_{30\delta}(A)$ and $|\sigma_2|\subset Q-N_{20\delta}(\partial Q)$. Let $H_t$ be the homotopy $H_t(x)=x+ty$. Then by \eqref{eq:homotopyformula} there exist homomorphisms $P: C_l(U)\rightarrow C_{l+1}(U_y)$ for all $l$ and all open sets $U\subset\R^n$, such that
\begin{equation}\label{eq:homotopia}
H_{1*}-H_{0*}=\partial P+P\partial.
\end{equation}
Note that $P(\partial\sigma_2)$ and $H_{1*}\sigma_2$ are chains in $Q-N_{10\delta}(\partial Q)$. We let $\sigma_y=\sigma_1-P(\partial \sigma_2)+H_{1*}\sigma_2$. Then $\sigma_y-\sigma=\partial P\sigma_2$ by \eqref{eq:homotopia}, so $\sigma_y$ belongs to the same relative homology class as $\sigma$. To prove the final part of the lemma, note that $|\partial\sigma_2|\subset N_{30\delta}(A)$, since $|\partial\sigma_2|=|\partial\sigma_1|\cap \mathrm{int}(Q)$. Thus $|P(\partial\sigma_2)|\subset N_{40\delta}(A)$ and $|\sigma_y|$, $|H_{1*}\sigma_2|=|\sigma_2|_y$ and $S_y$ all coincide outside $N_{100\delta}(A)$.  
\end{proof}
\begin{proof}[Proof of Lemma \ref{lemma:leikkauslemma}]
The lemma follows from the theory of intersection numbers developed in \cite{DoldAT}. We may assume that $Q=J^n$, where $J=[-1, 1]$, and respectively $A=\partial J^k\times J^{n-k}$ and $B=J^k\times \partial J^{n-k}$.  Let $\sigma_A$ and $\sigma_B$ be representatives of some nontrivial classes of $H_k(Q, A)$ and $H_{n-k}(Q, B)$, respectively. Suppose $|\sigma_A|\cap|\sigma_B|=\emptyset$. Then we can deform $\sigma_A$ and $\sigma_B$ slightly, if necessary, and assume that $|\sigma_A|\cap B=\emptyset=|\sigma_B|\cap A$. This allows us to define the \emph{intersection number} $[\sigma_A]\circ [\sigma_B]\in H_n(\R^n, \R^n-\{0\})\simeq \Z$ of the classes $[\sigma_A]$ and $[\sigma_B]$, as in \cite[VII.4]{DoldAT}. 

The intersection number of the two classes is defined (up to sign) by pushing the outer product 
\[
[\sigma_A]\times[\sigma_B]\in H_n(Q\times Q, A\times Q\cup Q\times B)
\]
forward with the map $u(x, y)=x-y$. Notice the analogy with the proof of Proposition \ref{lemma:taikalemma}. We do not describe the definition of the outer product here, as it is rather complicated and would take us too far away from the main topic.

Let us compute the intersection number by using two different pairs of representatives for $[\sigma_A]$ and $[\sigma_B]$. On one hand, since the images of the representatives $\sigma_A$ and $\sigma_B$ do not intersect, Propositions 4.5 and 4.6 of \cite[VII]{DoldAT} imply that $[\sigma_A]\circ[\sigma_B]=0$. On the other hand, $[\sigma_A]$ and $[\sigma_B]$ admit representatives that are integer multiples of triangulations of the subspaces $J^k\times \{0\}$ and $\{0\}\times J^{n-k}$, so combining Proposition 4.5 and Example 4.10 of \cite[VII]{DoldAT} shows that $[\sigma_A]\circ[\sigma_B]$ is nontrivial.
\end{proof}

\section{Lower bound and related open problems}\label{section:remarks}
Theorems \ref{thm:main} and \ref{thm:main2} raise the question:
\begin{question}\label{q:alaraja}
Do the lower bounds
\begin{equation}\label{eq:alaraja}
1\leq (\modd_p\Gamma_A)^{1/p}(\modd_{q}\Gamma_B)^{1/q}
\end{equation}
or 
\begin{equation}\label{eq:alaraja2}
1\leq (\modd_p\Gamma_A)^{1/p}(\modd_{q}\Gamma_A^*)^{1/q}
\end{equation}
hold whenever $Q, A$ and $B$ are as in Theorem \ref{thm:main}? 
\end{question}
Since $\Gamma_B\subset\Gamma_A^*$, \eqref{eq:alaraja} implies \eqref{eq:alaraja2}. All existing proofs, save the one in \cite{FreedmanHe1991}, of such lower bounds rely on some variation of the coarea formula, Lemma \ref{lemma:coarea}.

In \cite{FreedmanHe1991} a lower bound is proved for de Rham cohomology classes. Hence it may be possible to answer Question \ref{q:alaraja} by finding a connection between the moduli of $\Gamma_A$ and $\Gamma_B$, which can be seen as moduli of homology classes, and the moduli of suitable cohomology classes. This is of course easier said than done. For instance, it is not very clear what ``suitable cohomology'' should mean, when $Q$ is nonsmooth. It seems these kinds of questions are still largely unexplored.

Let us sketch a proof \eqref{eq:alaraja2} in the special case $k=1$. Then $A$ consists of two opposite faces $A_0$ and $A_1$ of $Q$ and, recalling the notation from the introduction, 
\[
\modd_p\Gamma_A=\modd_p\Gamma(A_0, A_1; Q).
\]
Moreover, by \cite{Shlyk1993}
\begin{equation}\label{eq:capacity}
\modd_p\Gamma(A_0, A_1; Q)=\capp_p\Gamma(A_0, A_1; Q),
\end{equation}
where the (Lipschitz) \emph{capacity} is defined by 
\[
\capp_p\Gamma(A_0, A_1; Q):=\inf_u\int_Q|\nabla u|^p\, d\Ha^n,
\]
and the infimum is taken over Lipschitz functions $u: Q\rightarrow I$ with $u|_{A_0}=0$ and $u|_{A_1}=1$. Then by the coarea formula
\[
1\leq\int_I\int_{u^{-1}(t)}\rho\, d\Ha^{n-1}dt=\int_Q\rho|\nabla u|\, d\Ha^n
\]
for any integrable $\rho$ admissible for $\Gamma_A^*$, since by \cite[3.2.15]{GMT} almost every level set $u^{-1}(t)$ is an element of $\Gamma_A^*$. Now the lower bound \eqref{eq:alaraja2} follows from H\"older's inequality and \eqref{eq:capacity}.

Similar ideas can be used to prove that Theorems \ref{thm:main} and \ref{thm:main2} are sharp for any $n$ and $k$. Let us show that \eqref{eq:alaraja} holds whenever $Q=Q_1\times Q_2$, where $Q_1\subset\R^k$ and $Q_2\subset\R^{n-k}$ are $k$- and ($n-k$)-dimensional topological cubes as in Theorem \ref{thm:main}, $A=\partial Q_1\times Q_2$ and $B=Q_1\times \partial Q_2$. Then it suffices to show that
\[
\modd_p\Gamma_A=\frac{\Ha^{n-k}(Q_2)}{\Ha^{k}(Q_1)^{p-1}}\,\text{ and }\,\modd_q\Gamma_B=\frac{\Ha^{k}(Q_1)}{\Ha^{n-k}(Q_2)^{q-1}}.
\]
The proofs of the two formulas are identical, so we only consider $\Gamma_A$. For every $y\in Q_2$ and $\rho$ admissible for $\Gamma_A$ 
\[
1\leq\int_{Q_1\times\{y\}}\rho\, d\Ha^{k},
\]
so by H\"older's inequality
\[
1\leq\left(\int_{Q_1\times\{y\}}\rho^p\, d\Ha^{k}\right)^{1/p}\Ha^{k}(Q_1)^{1/q},
\]
from which we obtain the inequality "$\geq$" by integrating over $y$ and applying Fubini's theorem (or the coarea formula applied on the projection $\pi_2(x, y)=y$). The reverse inequality follows from the observation that $\Ha^{k}(Q_1)^{-1}\chi_Q$ is admissible for $\Gamma_A$. 

It is also noteworthy that in this case $\modd_q\Gamma_B=\modd_q\Gamma_A^*$, and both are equal to the $q$-modulus of the slices $\{x\}\times Q_2$. 

Observe that if we let $\lambda=\Ha^k(Q_1)^{-1/k}$ and use a scaled projection map $\lambda\pi_1(x, y)=\lambda x$ instead, we find that $\Ha^k(\lambda \pi_1(Q_1\times Q_2))=1$ and $J_{\lambda\pi_1}=\Ha^k(Q_1)^{-1}\chi_Q$. That is, the minimizer of $\modd_p\Gamma_A$ is the jacobian of $\lambda\pi_1$. Moreover, the level sets of $\lambda\pi_1$ are elements of $\Gamma_B$.

Inspired by this example we extend the definition of the capacity to general $Q$ and $A$ by
\[
\capp_p\Gamma_A:=\inf_u\int_QJ_u^p\, d\Ha^n,
\]
where the infimum is taken over all such Lipschitz maps $u: (Q, A)\rightarrow (\bar U, \partial U)$, that $U$ is a domain in $\R^k$ normalized with $\Ha^k(U)=1$, $(\bar U, \partial U)$ is homeomorphic to $(\bar \B^k, \partial\B^k)$, and the induced homomorphism
\begin{equation}\label{eq:isomorfismi}
u_*: H_k(Q, A)\rightarrow H_k(\bar U, \partial U)\simeq\Z
\end{equation}
is an isomorphism. We observe that $U\subset u(S)$ for any $S\in\Gamma_A$, so almost every level set of $u$ is in $\Gamma_A^*$, since $H_k(\bar U-\{x\}, \partial U)$ is trivial for all $x\in U$. Moreover, the Cauchy-Binet formula implies that $J_u\geq J_u^S$, so 
\[
\int_SJ_u\, d\Ha^k\geq\int_S J_u^S\, d\Ha^k\geq\int_U\, d\Ha^k=1
\]
by Lemma \ref{lemma:coarea}. Thus $J_u$ is admissible for $\Gamma_A$ and
\[
\modd_p\Gamma_A\leq \capp_p\Gamma_A.
\]
It is unknown whether the reverse inequality is true, but it would imply \eqref{eq:alaraja2}. To prove the reverse inequality one would have to be able to construct the required Lipschitz maps $u$. This seems to be very difficult when $k>1$, especially with a given $J_u$. If $k=1$, the situation is considerably simpler, since then $J_u=|\nabla u|$ and the unit interval $I$ is practically the only choice of $U$.
\bibliographystyle{plain}
\bibliography{bibbi}
\vspace{1em}
\noindent
Department of Mathematics and Statistics, University of Jyv\"askyl\"a, P.O.
Box 35 (MaD), FI-40014, University of Jyv\"askyl\"a, Finland.\\

\emph{E-mail:} \settowidth{\hangindent}{\emph{aaaaaaaaa}} \textbf{atte.s.lohvansuu@jyu.fi} 
\end{document}